\documentclass{sig-alternate}
\usepackage{microtype}
\usepackage{epstopdf}
\usepackage{graphicx} 
\usepackage{subfigure} 
\graphicspath{{.}{./figures/}}

\usepackage{natbib}
\renewcommand{\cite}{\citep}

\usepackage{algorithm}
\usepackage{algorithmic}

\usepackage[pdfpagelabels=false,plainpages=false,colorlinks,citecolor=black,linkcolor=black,urlcolor=black]{hyperref}

\usepackage{tabularx}
\usepackage{amsmath}
\usepackage{amssymb}
\usepackage{paralist}
\usepackage{booktabs}
\usepackage{listings}
\newtheorem{theorem}{Theorem}
\newtheorem{lemma}[theorem]{Lemma}
\newtheorem{definition}[theorem]{Definition}

\newcommand{\sA}{\mathcal{A}}

\DeclareMathOperator{\rank}{rank}
\DeclareMathOperator{\tr}{trace}
\DeclareMathOperator{\sign}{sign}

\newcommand{\algo}[1]{\textsc{\lowercase{#1}}}

\newenvironment{sbmatrix}{\left[\!\begin{smallmatrix}}
{\end{smallmatrix}\!\right]}

\newcommand{\itn}[1]{^{(#1)}}
\newcommand{\eps}{\varepsilon}
\newcommand{\mat}{\boldsymbol}
\renewcommand{\vec}[1]{\boldsymbol{\mathrm{#1}}}
\providecommand{\mYhat}{\ensuremath{\mat{\hat{\mY}}}}
\providecommand{\mSigma}{\ensuremath{\mat{\Sigma}}}
\providecommand{\mA}{\ensuremath{\mat{A}}}
\providecommand{\mB}{\ensuremath{\mat{B}}}

\providecommand{\mD}{\ensuremath{\mat{D}}}
\providecommand{\mE}{\ensuremath{\mat{E}}}

\providecommand{\mK}{\ensuremath{\mat{K}}}

\providecommand{\mR}{\ensuremath{\mat{R}}}
\providecommand{\mS}{\ensuremath{\mat{S}}}
\providecommand{\mT}{\ensuremath{\mat{T}}}
\providecommand{\mU}{\ensuremath{\mat{U}}}
\providecommand{\mV}{\ensuremath{\mat{V}}}
\providecommand{\mW}{\ensuremath{\mat{W}}}
\providecommand{\mX}{\ensuremath{\mat{X}}}
\providecommand{\mY}{\ensuremath{\mat{Y}}}

\providecommand{\vb}{\ensuremath{\vec{b}}}

\providecommand{\ve}{\ensuremath{\vec{e}}}

\providecommand{\vs}{\ensuremath{\vec{s}}}

\providecommand{\vx}{\ensuremath{\vec{x}}}

\providecommand{\vz}{\ensuremath{\vec{z}}}
 
\DeclareMathOperator*{\minimize}{minimize}

\DeclareMathOperator{\subjectto}{subject\ to}
\providecommand{\MINof}[1][]{{\displaystyle \minimize_{#1}}}
\providecommand{\MIN}[2][]{\begin{array}{ll} \MINof[#1] & #2 \\ \end{array}}
\providecommand{\MINone}[3]{\begin{array}{ll} \MINof[#1] & #2 \\ \subjectto  & #3 \end{array}}
\providecommand{\MINtwo}[4]{\begin{array}{ll} \MINof[#1] & #2 \\ \subjectto  & #3 \\ & #4 \end{array}}
\newcommand{\sstretchsym}[3]{\ensuremath{\left#1 #3 \right#2}}
\newcommand{\nstretchsym}[3]{\ensuremath{#1 #3 #2}}
\newcommand{\bstretchsym}[3]{\ensuremath{\bigl#1 #3 \bigr#2}}
\newcommand{\Bstretchsym}[3]{\ensuremath{\Bigl#1 #3 \Bigr#2}}
\newcommand{\hstretchsym}[3]{\ensuremath{\biggl#1 #3 \biggr#2}}
\newcommand{\Hstretchsym}[3]{\ensuremath{\Biggl#1 #3 \Biggr#2}}
\newcommand{\newstrechsymset}[4][s]{%
	\expandafter\def\csname s#2of\endcsname{\sstretchsym{#3}{#4}}%
	\expandafter\def\csname n#2of\endcsname{\nstretchsym{#3}{#4}}%
	\expandafter\def\csname b#2of\endcsname{\bstretchsym{#3}{#4}}%
	\expandafter\def\csname B#2of\endcsname{\Bstretchsym{#3}{#4}}%
	\expandafter\def\csname h#2of\endcsname{\hstretchsym{#3}{#4}}%
	\expandafter\def\csname H#2of\endcsname{\Hstretchsym{#3}{#4}}%
	\expandafter\def\csname #2of\endcsname{\csname #1#2of\endcsname}%
}
\newcommand{\normof}[2][]{\sstretchsym{\|}{\|}{#2}_{#1}}
\newcommand{\nnormof}[2][]{\nstretchsym{\|}{\|}{#2}_{#1}}
\newstrechsymset{abs}{|}{|}

\def\clap#1{\hbox to 0pt{\hss#1\hss}}

\def\mathclap{\mathpalette\mathclapinternal}

\def\mathclapinternal#1#2{%
  \clap{$\mathsurround=0pt#1{#2}$}}

\begin{document} 

\title{Rank Aggregation via Nuclear Norm Minimization}

\numberofauthors{2} 

\author{
%
%
\alignauthor
David F.~Gleich\\
       \affaddr{Sandia National Laboratories\thanks{{\crnotice
       Sandia National Laboratories is a multi-program laboratory
       managed and operated by Sandia Corporation, a wholly owned
       subsidiary of Lockheed Martin Corporation, for the U.S.
       Department of Energy's National Nuclear Security Administration
       under contract DE-AC04-94AL85000.}}}\\
       \affaddr{Livermore, CA}\\
       \email{dfgleic@sandia.gov}
\alignauthor
Lek-Heng Lim\\
			 \affaddr{University of Chicago}\\
			 \affaddr{Chicago, IL}\\
			 \email{lekheng@galton.chicago.edu}
}

\maketitle

\begin{abstract} 
The process of rank aggregation is intimately intertwined with
the structure of skew-symmetric matrices.  
We apply recent advances in the theory and algorithms of matrix completion
to skew-symmetric matrices.  This combination of ideas
produces a new method for ranking a set of items.  The essence
of our idea is that a rank aggregation describes a partially
filled skew-symmetric matrix.  We extend an algorithm for
matrix completion to handle skew-symmetric data and use that
to extract ranks for each item.  
Our algorithm applies to both pairwise comparison and rating data.
Because it is based on matrix completion, it is robust to
both noise and incomplete data.  We show a formal
recovery result for the noiseless case and 
present a detailed study of the algorithm on synthetic
data and Netflix ratings.
\end{abstract} 

\keywords{nuclear norm, skew symmetric, rank aggregation}

\section{Introduction}



One of the classic data mining problems is to identify the
important items in a data set; see \citet{Tan-2004-ordering}
for an interesting example of how these might be used. For this task, 
we are concerned with rank aggregation.
Given a series of votes on a set of items by a group of voters,
rank aggregation is the process of permuting the set of items so that
the first element is the best choice in the set, the second
element is the next best choice, and so on.  
In fact, rank aggregation is an old problem and has a history stretching 
back centuries~\cite{condorcet1785-essai}; one famous 
result is that any rank aggregation requires some 
degree of compromise~\cite{arrow1950-impossibility}.  
Our point in this introduction
is not to detail a history of all the possible methods of
rank aggregation, but to give some perspective on our approach to
the problem.

Direct approaches involve finding a permutation
explicitly -- for example, computing the Kemeny optimal
ranking \cite{kemeny1959-math-without-numbers} or
the minimum feedback arc set problem.  
These problems are NP-hard 
\cite{dwork2001-rank-aggregation,ailon2005-ranking,%
alon2006-ranking}.
An alternate approach is to assign a score to 
each item, and then compute a permutation based on
ordering these items by their score, e.g.\ \citet{saaty1987-perron}.
In this manuscript, we focus on the second approach.  A
key advantage of the computations we propose
is that they are \emph{convex} problems and efficiently
solvable.

While
the problem of rank aggregation is old, modern applications --
such as those found in web-applications like Netflix and Amazon --
pose new challenges. 
First, the data collected are usually
cardinal measurements on the quality of each
item, such as  1--5 stars, received from voters.
Second, the voters are neither experts in the rating domain nor
experts at producing useful ratings.  These properties
manifest themselves in a few ways, including skewed
and indiscriminate voting behaviors \cite{ho2008-ratings}.
We focus on using aggregate pairwise 
data about items to develop a score for each item that
predicts the pairwise data itself.  This approach
eliminates some of the issues with directly utilizing
voters ratings, and we argue this point more precisely
in Section~\ref{sec:pairwise}.


To explain our method, consider a set of $n$ items, labeled from $1$ to $n$.  
Suppose that each of these
items has an unknown intrinsic quality $s_i : 1 \le i \le n$,
where $s_i > s_j$ implies that item $i$ is better
than item $j$.  While the $s_i$'s are unknown, 
suppose we are given
a matrix $\mY$ where $Y_{ij} = s_i - s_j$.  By 
finding a rank-2 factorization of $\mY$, for example
\begin{equation}
\mY = \vs \ve^T - \ve \vs^T,
\end{equation}
we can extract unknown scores.  The matrix $\mY$
is  skew-symmetric and describes any score-based global pairwise
ranking.  (There are other possible rank-2 factorizations
of a skew-symmetric matrix, 
a point we return to later in Section~\ref{sec:mc-algs}).


Thus, given a measured $\mYhat$, the goal is to find
a minimum rank approximation of $\mYhat$ that models the elements, and ideally
one that is rank-$2$.  Phrased
in this way, it is a natural candidate for recent developments in the
theory of matrix completion~\cite{candes2009-matrix-completion,recht2009-nuclear-norm}.
In the matrix
 completion problem, certain elements of the matrix are presumed to be known.
 The goal is to produce a low-rank matrix that respects these elements -- or at
  least minimizes the deviation from the known elements.  One catch, 
 however, is that we require
 matrix completion over skew-symmetric matrices for pairwise
ranking matrices.  Thus, we must
solve the matrix completion problem inside a structured class of 
matrices.  This task is a novel contribution of our work.
Recently, \citet{Gross2010-low-rank} also developed a technique
for matrix completion with Hermitian matrices.

With a ``completed'' matrix $\mY$, the norm of the residual $\nnormof{ \mYhat - \mY }$
gives us a certificate for the validity of our fit -- an additional piece of
information available in this model.

To continue, we briefly summarize our main contributions and 
our notational conventions.  

\paragraph{Our contributions}
\begin{compactitem}
 \item We propose a new method for computing a rank aggregation based on 
 matrix completion, which is tolerant to noise and incomplete data.
 \item We solve a structured matrix-completion problem
 over the space  of skew-symmetric matrices.
 \item We prove a recovery theorem detailing when our approach
 will work.
 \item We perform a detailed evaluation of our approach with
   synthetic data and an anecdotal study with Netflix ratings.
\end{compactitem}

\paragraph{Notation} We try to follow standard
notation conventions.  Matrices are bold, upright roman letters, vectors
are bold, lowercase roman letters, and scalars are unbolded roman or Greek
letters.  The vector $\ve$ consists of all ones,
and the vector $\ve_i$ has a $1$ in the $i$th position and $0$'s
elsewhere.
Linear maps on matrices are written as script letters.  
An index set $\Omega$
is a group of index pairs.  Each $\omega \in \Omega$
is a pair $(r,s)$ and we assume that the $\omega$'s are
numbered arbitrarily, i.e. $\Omega = \{ \omega_1, \ldots, \omega_k \}$.
Please refer to Table~\ref{tab:notation} for reference.

\begin{table}
\caption{Notation for the paper.}
\label{tab:notation}
\setlength{\tabcolsep}{1ex}
\renewcommand{\arraystretch}{1.2}
\begin{tabularx}{\linewidth}{cX}
\toprule
  \textbf{Sym.} & \textbf{Interpretation} 
\\ \midrule
   $\sA(\cdot)$ & a linear map from a matrix to a vector 
\\ $\ve$ & a vector of all ones
\\ $\ve_i$ & a vector with $1$ in the $i$th entry, 0 elsewhere
\\ $\normof[*]{\cdot}$ & the nuclear norm
\\ $\mR$ & a rating matrix (voters-by-items)
\\ $\mY$ & a fitted or model pairwise comparison matrix 
\\ $\mYhat$ & a measured pairwise comparison matrix
\\ $\Omega$ & an index set for the known entries of a matrix
\\ \bottomrule
\end{tabularx}
\end{table}

\begin{table*}
\caption{The top 15 movies from Netflix generated by our ranking method 
(middle and right).  The left list is the ranking using the
mean rating of each movie and is
emblematic of the problems global ranking methods face when
infrequently compared items rocket to the top.  We prefer the middle and
right lists.  See Section~\ref{sec:results} and Figure~\ref{fig:netflix-residuals}
for information about the conditions and additional discussion.  
LOTR III appears twice because of the 
two DVDs editions, theatrical and extended.}
\label{tab:netflix-top-movies}
\begin{tabularx}{\linewidth}{XXX}
\toprule
		Mean	& 		Log-odds (all)	&		Arithmetic Mean (30)
\\ \midrule		
		  LOTR III: Return \ldots	&		LOTR III: Return \ldots	&		LOTR III: Return \ldots
\\		LOTR I: The Fellowship \ldots	&		LOTR I: The Fellowship \ldots	&		LOTR I: The Fellowship \ldots
\\		LOTR II: The Two \ldots	&		LOTR II: The Two \ldots	&		LOTR II: The Two \ldots
\\		Lost: Season 1	&		Star Wars V: Empire \ldots	&		Lost: S1
\\		Battlestar Galactica: S1	&		Raiders of the Lost Ark	&		Star Wars V: Empire \ldots
\\		Fullmetal Alchemist	&		Star Wars IV: A New Hope	&		Battlestar Galactica: S1
\\		Trailer Park Boys: S4	&		Shawshank Redemption	&		Star Wars IV: A New Hope
\\		Trailer Park Boys: S3	&		Star Wars VI: Return ...	&		LOTR III: Return \ldots
\\		Tenchi Muyo! \ldots	&		LOTR III: Return \ldots	&		Raiders of the Lost Ark
\\		Shawshank Redemption	&		The Godfather	&		The Godfather
\\		Veronica Mars: S1	&		Toy Story	&		Shawshank Redemption
\\		Ghost in the Shell: S2	&		Lost: S1	&		Star Wars VI: Return ...
\\		Arrested Development: S2	&		Schindler's List	&		Gladiator
\\		Simpsons: S6	&		Finding Nemo	&		Simpsons: S5
\\		Inu-Yasha	&		CSI: S4	&		Schindler's List
\\ \bottomrule
\end{tabularx}
\end{table*}

Before proceeding further, let us outline the rest of the paper.
First, Section~\ref{sec:pairwise} describes a few 
methods to take voter-item ratings and produce an aggregate 
pairwise comparison matrix.  Additionally, we argue 
why pairwise aggregation is
a superior technique when the goal is to produce
an ordered list of the alternatives.  
Next, in Section~\ref{sec:ranking-nn}, we describe  
formulations of the noisy matrix completion problem 
using the nuclear norm.
In our setting, the 
\algo{lasso} formulation is the best choice, and we use
it throughout the remainder.  We 
briefly describe algorithms for matrix completion and focus
on the \algo{svp} algorithm  \cite{Jain-2010-SVP} 
in Section~\ref{sec:mc-algs}.
We then show that the \algo{svp} algorithm preserves skew-symmetric structure.
This process
involves studying the singular value decomposition of skew-symmetric
matrices.  Thus, by the end of the section, we've shown how to
formulate and solve for a scoring vector based on the nuclear norm.
The following sections describe alternative approaches and show our
recovery results.  At the end, we show our experimental results.
In summary, our overall methodology is 
\begin{center}
Ratings \rlap{($= \mR$)} \\
$\Downarrow$ \rlap{(\S \ref{sec:pairwise})}\\
 Pairwise comparisons \rlap{($= \mY$)}\\
$\Downarrow$ \rlap{(\S \ref{sec:ranking-nn})}\\
Ranking scores \rlap{($= \vs$)} \\
$\Downarrow$ \rlap{(sorting)} \\
Rank aggregations.
\end{center}
An example of our rank aggregations 
is given in Table~\ref{tab:netflix-top-movies}.  We comment
further on these in Section~\ref{sec:netflix}.

Finally, we provide our computational and experimental
codes so that others may reproduce our results:\\
\url{https://dgleich.com/projects/skew-nuclear}


\section{Pairwise Aggregation Methods}
\label{sec:pairwise}

To begin, we describe methods to aggregate the votes
of many voters, given by the matrix $\mR$, 
into a measured pairwise comparison matrix $\mYhat$.  
These methods have been well-studied in statistics
\cite{david1988-paired}.
In the next
section, we show how to extract a score for each item from the 
matrix $\mYhat$.

Let $\mR$ be a voter-by-item matrix.  This matrix has $m$ rows corresponding
to each of the $m$ voters and $n$ columns corresponding to the $n$ items
of the dataset.  In all of the applications we explore, the matrix 
$\mR$ is highly incomplete.  That is, only a few items are rated by
each voter.  Usually all the items have a few votes, but there is
no consistency in the number of ratings per item.

Instead of using $\mR$ directly, we compute a pairwise aggregation.  
Pairwise comparisons have a lengthy history, dating back to the 
first half of the previous century \cite{Kendall-1940-paired-comparison}.
They also have many nice properties.
First, \citet{miller1956-seven} observes that most people can evaluate only
5 to 9 alternatives at a time.  This fact may relate to the common choice
of a $5$-star rating (e.g.\ the ones used by Amazon, eBay, 
Netflix, YouTube).  Thus, comparing pairs of
movies is easier than ranking a set of $20$ movies.  Furthermore,
only pairwise comparisons are possible in certain settings
such as tennis tournaments. 
Pairwise comparison methods are thus natural for analyzing ranking data.
Second, pairwise comparisons are a relative measure and help reduce 
bias from the rating scale. For these reasons, pairwise
comparison methods have been popular in psychology, statistics, and social
choice theory \cite{david1988-paired,arrow1950-impossibility}. Such
methods have also been adopted by the learning to rank community;
see the contents of \citet{li2008-learning-to-rank}.
A final advantage of pairwise methods is that they are much more complete
than the ratings matrix.  For Netflix, $\mR$ is 99\% incomplete, whereas
$\mY$ is only 0.22\% incomplete and most entries are supported by
\emph{many} comparisons.  See Figure~\ref{fig:pairwise} for information
about the number of pairwise comparisons in Netflix and MovieLens.

More critically, an incomplete array of user-by-product ratings
is a strange matrix -- not every 2-dimensional array of
numbers is best viewed as a matrix --
and using the rank of this matrix (or its convex relaxation) as a key
feature in the modeling needs to be done with care.  Consider,
if instead of rating values 1 to 5, 0 to 4 are used to
represent the exact same information, the rank of this
new rating matrix will change.
Furthermore, whether we use a rating scale where 1 is the best
rating and 5 is worst, or one where 5 is the best
and 1 is the worst, a low-rank model would give the
exact same fit with the same input values, even
though the connotations of the numbers is reversed.

On the other hand, the pairwise ranking matrix that
we construct below is invariant under monotone transformation
of the rating values and depends only on the degree of
relative preference of one alternative over another.
It circumvents the previously mentioned pitfalls and is
a more principled way to employ a rank/nuclear norm model.

We now describe five techniques to build an aggregate pairwise
matrix $\mYhat$ from the rating matrix $\mR$.  
Let $\alpha$ denote the index of a voter, and $i$ and $j$ the 
indices of two items.  The entries of $\mR$ are $R_{\alpha i}$.
To each voter, we associate a pairwise comparison matrix $\mYhat^\alpha$.
The aggregation
is usually computed by something like a mean over $\mYhat^\alpha$.

\begin{compactenum}
\item \textbf{Arithmetic mean of score differences}\quad The score difference
is $Y^{\alpha}_{ij} = R_{\alpha j}-R_{\alpha i}$. The arithmetic mean
of all voters who have rated both $i$ and $j$ is
\[
\hat{Y}_{ij}=\frac{\sum_{\alpha}(R_{\alpha i}-R_{\alpha j})}{\#\{\alpha\mid
R_{\alpha i},R_{\alpha j}\text{ exist}\}}.
\]
These comparisons  are translation invariant.

\item \textbf{Geometric mean of score ratios}\quad Assuming $\mR>0$, the score ratio
refers to $Y^{\alpha}_{ij} = R_{\alpha j}/R_{\alpha i}$. The (log) geometric
mean over all voters who have rated both $i$ and $j$ is
\[
\hat{Y}_{ij}=\frac{\sum_{\alpha}(\log R_{\alpha i} -\log R_{\alpha j}
)}{\#\{\alpha\mid R_{\alpha i},R_{\alpha j}\text{ exist}\}}.
\]
These are  scale invariant.

\item \textbf{Binary comparison}\quad Here $Y^{\alpha}_{ij} = \operatorname*{sign}%
(R_{\alpha j} - R_{\alpha i} )$. Its average is the probability difference
that the alternative $j$ is preferred to $i$ than vice versa
\[
\hat{Y}_{ij}=\Pr\{\alpha\mid R_{\alpha i}>R_{\alpha k}\}-\Pr\{\alpha\mid
R_{\alpha i}<R_{\alpha j}\}.
\]
These are invariant to a monotone transformation.

\item \textbf{Strict binary comparison}\quad This method is almost the same
as the last method, except that we eliminate cases where users rated
movies equally.  That is, 
\[ \hat{Y}^\alpha_{ij} = \begin{cases} 1 & R_{\alpha i} > R_{\alpha j} \\
                                 - & R_{\alpha i} = R_{\alpha j} \\
                                 -1 & R_{\alpha i} < R_{\alpha j}.
                   \end{cases} \]
Again, the average $Y_{ij}$ has a similar interpretation to
binary comparison, but only among people who expressed a strict 
preference for one item over the other.  
Equal ratings are ignored.

\item \textbf{Logarithmic odds ratio}\quad This idea translates
binary comparison to a logarithmic scale:
\[
\hat{Y}_{ij}=\log\frac{\Pr\{\alpha\mid R_{\alpha i}\geq R_{\alpha j}\}}%
{\Pr\{\alpha\mid R_{\alpha i}\leq R_{\alpha j}\}}.
\]
\end{compactenum}

\begin{figure*}
\centering
\subfigure[MovieLens - 85.49\% of total pairwise comparisons]{%
  \includegraphics[width=0.45\linewidth]{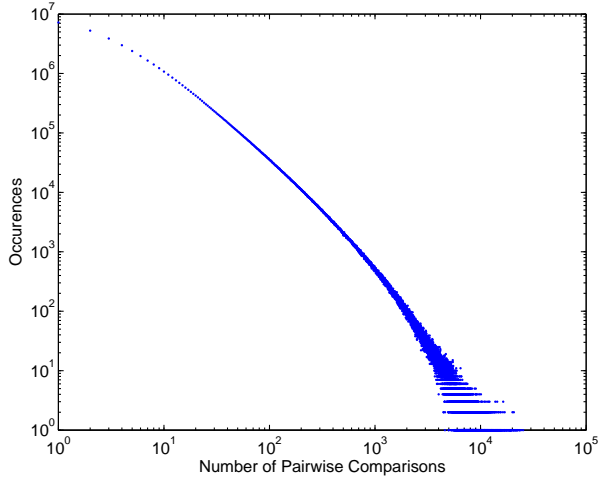}%
}  
\quad 
\subfigure[Netflix - 99.77\% of total pairwise comparisons]{%
  \includegraphics[width=0.45\linewidth]{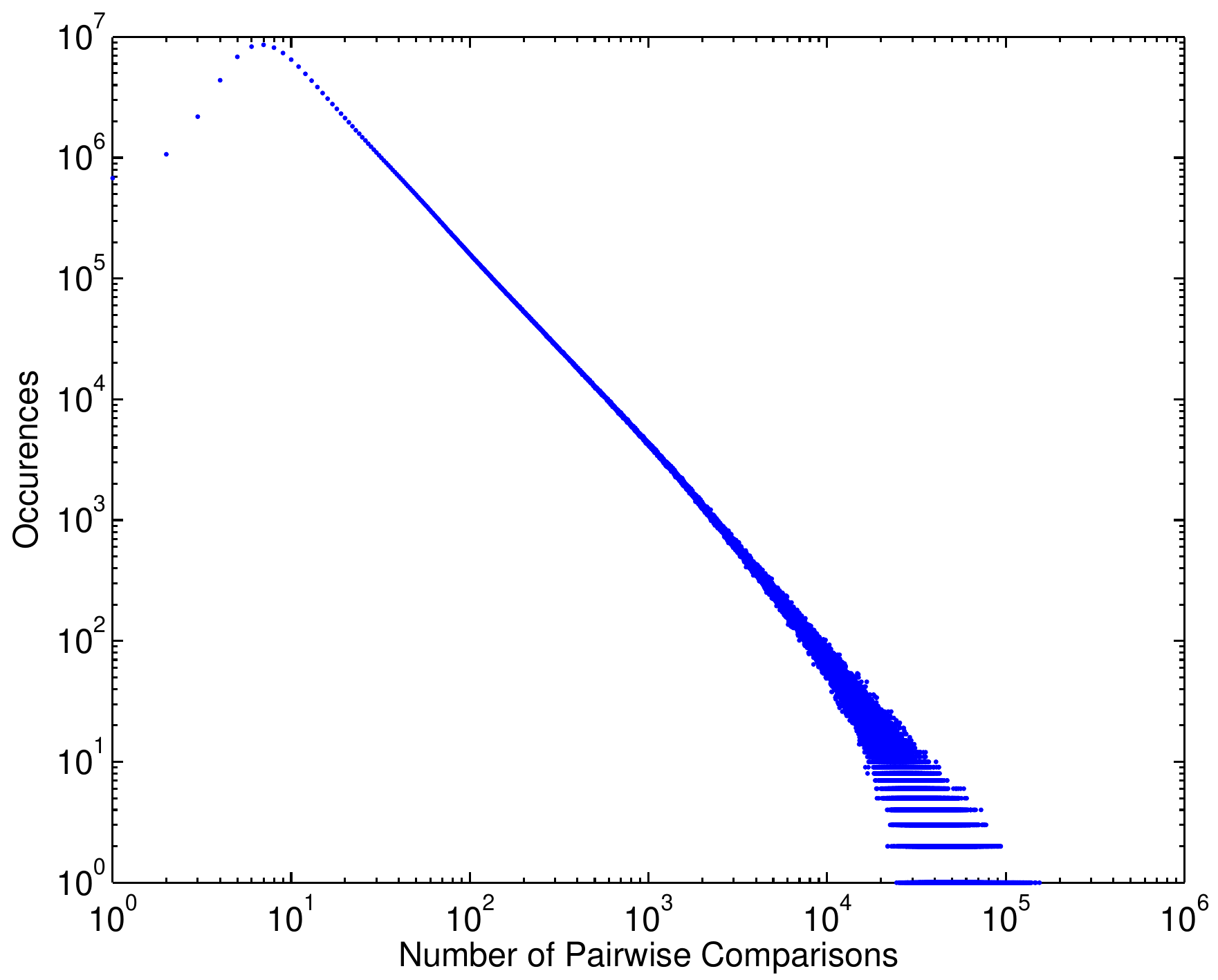}%
}  
\caption{A histogram of the number of pairwise comparisons 
between movies in MovieLens (left) and Netflix (right).  The 
number of pairwise comparisons is the number of users with ratings
on both movies.  These histograms show that most items
have more than a small number of comparisons between them.  
For example, 18.5\% and 34.67\% of all possible pairwise
entries have more than 30 
comparisons between them.  Largely speaking, this figure
justifies dropping infrequent ratings from the comparison.
This step allows us to take advantage of the ability
of the matrix-completion methods to deal with incomplete
data. }
\label{fig:pairwise}
\end{figure*}

\section{Rank Aggregation with the\\ Nuclear Norm} 
\label{sec:ranking-nn}

Thus far, we have seen how to compute an aggregate
pairwise matrix $\mYhat$ from ratings data.
While $\mYhat$ has fewer missing entries than $\mR$ --
roughly 1-80\% missing instead of almost 99\% missing
-- it is still not nearly complete.  
In this section, we discuss how to use the theory
of matrix completion to estimate the scoring
vector underlying the comparison matrix $\mYhat$.
These same techniques apply even when $\mYhat$ is not
computed from ratings and is measured through
direct pairwise comparisons.  

Let us now state the matrix completion problem formally
\cite{candes2009-exact-completion,recht2009-nuclear-norm}.
Given a matrix $\mA$ where only a subset of the entries
are known, the goal is to find the 
lowest rank matrix $\mX$ that agrees with $\mA$ in all
the non-zeros.  
Let $\Omega$ be the index set corresponding
to the known entries of $\mA$.
Now define $\sA(\mX)$ as a linear map corresponding to the
elements of $\Omega$, i.e. $\sA(\mX)$ is a vector where
the $i$th element is defined to be
\begin{equation}
[\sA(\mX)]_i = X_{\omega_i},
\end{equation}
and where we interpret $X_{\omega_i}$ as the entry of the 
matrix $\mX$ for the index pair $(r,s) = \omega_i$.
Finally, let $\vb = \sA(\mY)$ be the values of the specified 
entries of the matrix $\mY$.  
This idea of matrix completion corresponds with the
solution of 
\begin{equation} \label{eq:mc-rank}
\MINone{}{\rank(\mX)}{\sA(\mX) = \vb.}
\end{equation}
Unfortunately, like the direct
methods at permutation minimization, this approach 
is NP-hard \cite{vandenberghe1996-semidefinite}.

To make the problem tractable, an increasingly well-known technique is
to replace the $\rank$ function with the nuclear norm
 \cite{fazel2002-phdthesis}.
For a matrix $\mA$, the nuclear norm is defined 
\begin{equation}
\normof[*]{\mA} = \sum_{i=1}^{\mathclap{\rank(\mA)}} \sigma_i(\mA)
\end{equation}
where $\sigma_i(\mA)$ is the $i$th singular value of
$\mA$.
The nuclear norm has a few other names:
the Ky-Fan $n$-norm, the Schatten $1$-norm,
and the trace norm (when applied to symmetric matrices),
but we will just use the term nuclear norm here.
It is a convex underestimator of the
rank function on the unit spectral-norm ball $\{ \mA : \sigma_{\max}(\mA) \le 1 \}$, i.e.\ $\nnormof[*]{\mA} \le \rank(\mA) \sigma_{\max}(\mA)$
and is the largest convex function with this property.
Because the nuclear norm is convex,
\begin{equation} \label{eq:mc-nn}
\MINone{}{\normof[*]{\mX}}{\sA(\mX) = \vb}
\end{equation}
is a convex relaxation of \eqref{eq:mc-rank}
analogous to how the $1$-norm is a convex
relaxation of the $0$-norm.  

In \eqref{eq:mc-nn}
we have $\sA(\mX) = \vb$, which is called a
noiseless completion problem.  Noisy completion
problems only require $\sA(\mX) \approx \vb$.  
We present four possibilities inspired by similar
approaches in compressed sensing.
For the compressed sensing problem with noise:
\[
\minimize \nnormof[1]{\vx} \quad \subjectto \; \mA \vx \approx \vb
\]
there are four well known formulations: 
\algo{Lasso} \cite{tibshirani1996-lasso}, 
\algo{QP}  \cite{chen1998-atomic}, 
\algo{DS}  \cite{candes2007-dantzig}
and \algo{BPDN}  \cite{fuchs2004-noise}.  
For the noisy matrix completion problem, the same variations apply,
but with the nuclear norm taking the place of the $1$-norm:
\begin{center}
\begin{minipage}{\linewidth}
\makebox[4em][l]{\textsc{lasso}} \\[1ex]
$\MINone{}{\normof[2]{\sA(\mX) - \vb}}{\normof[*]{\mX} \le \tau}$
\end{minipage}\\[2ex]
\begin{minipage}{\linewidth}
\makebox[4em][l]{\textsc{ds}}  \\[1ex]
$\MINone{}{\normof[*]{\mX}}{\sigma_{\max}(\sA^*(\sA(\mX) - \vb))\le \mu}$
\end{minipage}\\[2ex]
\begin{minipage}{\linewidth}
\makebox[4em][l]{\textsc{qp}} \citet{mazumder2009-regularization} \\[1ex]
$\MIN{}{\normof[2]{\sA(\mX) - \vb}^2 + \lambda \normof[*]{\mX}}$
\end{minipage}\\[2ex]
\begin{minipage}{\linewidth}
\makebox[4em][l]{\textsc{bpdn}} \citet{mazumder2009-regularization} \\[1ex]
$\MINone{}{\normof[*]{\mX}}{\normof[2]{\sA(\mX) - \vb}\le \sigma}$
\end{minipage}
\end{center}

Returning to rank-aggregation, 
recall the perfect case for the matrix $\mY$:
there is an unknown quality $s_i$
associated with each item $i$ and
 $\mY = \vs \ve^T - \ve \vs^T$.
We now assume that the pairwise comparison matrix
computed in the previous section approximates 
the true $\mY$.  Given such a $\mYhat$, our goal is to complete
it with a rank-2 matrix.  Thus, our objective:
\begin{equation} \label{eq:ranking-mc}
 \MINtwo{}{\normof[2]{\sA(\mX) - \vb}}{\normof[*]{\mX} \le 2}{\mX = -\mX^T}
\end{equation}
where $\sA(\cdot)$ corresponds to the filled entries of $\mYhat$.
We adopt the \algo{lasso} formulation because we want
$\rank(\mX) = 2$, and $\nnormof[*]{\mX}$ underestimates
rank as previously mentioned.
This problem only differs from the standard matrix completion 
problem in one regard: the skew-symmetric constraint.
With a careful choice of solver, this
additional constraint comes ``for-free'' (with a few technical
caveats).  It should also be possible
to use the skew-Lanczos process to exploit
the skew-symmetry in the SVD computation.  

\subsection{Algorithms} \label{sec:mc-algs}

Algorithms for matrix completion
seem to sprout like wildflowers in spring:
\citet{lee2009-admira,cai2008-svt,%
toh2009-proximal-gradient,%
dai2009-set,keshavan2009-grassman,mazumder2009-regularization,%
Jain-2010-SVP}.  Each algorithm
fills a slightly different niche, or improves a performance measure 
compared to its predecessors.  

We first explored crafting our own solver by adapting projection and
thresholding ideas used in these algorithms to the skew-symmetrically
constrained variant.  However, we realized that many algorithms
do not require any modification to solve the problem with the 
skew-symmetric constraint.  This result follows from
properties of skew-symmetric matrices we show below.

Thus, we use the \algo{svp} algorithm by \citet{Jain-2010-SVP}.  
For the matrix completion problem, they found their implementation
outperformed many competitors.  It is scalable
and handles a \algo{lasso}-like objective for a fixed rank approximation.
For completeness, we restate the \algo{svp} procedure in Algorithm~\ref{fig:svp}.

\begin{algorithm} \raggedright
 \caption{Singular Value Projection \cite{Jain-2010-SVP}:
 Solve a matrix completion problem.  We use
  the notation $\Omega(\mX)$ to denote output of 
  $\sA(\mX)$ when $\sA(\cdot)$ is an index set.}
 \label{fig:svp}
 \begin{algorithmic}[1]
  \REQUIRE index set $\Omega$, target values $\vb$, target rank $k$, maximum rank $k$, step length $\eta$, tolerance $\eps$
  \STATE Initialize $\mX\itn{0} = 0$, $t=0$
  \REPEAT 
    \STATE Set $ \mU\itn{t} \mSigma\itn{t} {\mV\itn{t}}^T$ to be 
        the rank-$k$ SVD of a matrix with 
        non-zeros $\Omega$ and values\\ 
        $\Omega(\mX\itn{t}) - \eta (\Omega(\mX\itn{t}) - \vb)$
    \STATE $\mX\itn{t+1} \leftarrow \mU\itn{t} \mSigma\itn{t} {\mV\itn{t}}^T$
    \STATE $t \leftarrow t+1$
  \UNTIL $\nnormof[2]{\Omega(\mX\itn{k}) - \vb} > \eps$
 \end{algorithmic}

\end{algorithm}

\label{sec:mc-algs-ss}

If the constraint $\sA(\mX), \vb$
comes from a skew-symmetric matrix,
then this algorithm produces a skew-symmetric
matrix as well.  Showing this involves
a few properties of skew-symmetric matrices and two lemmas.

We begin by stating a few well-known properties of skew-symmetric
matrices. Let $\mA=-\mA^T$ be skew-symmetric.
Then all the eigenvalues of 
$\mA$ are pure-imaginary and come in complex-conjugate pairs. 
Thus, a skew-symmetric matrix must always have even rank.
Let $\mB$ be a square real-valued matrix, then the closest
skew-symmetric matrix to $\mB$ (in any norm) is 
$\mA = (\mB - \mB^T)/2$.  These results have
elementary proofs.  We continue by characterizing the
singular value decomposition of a skew-symmetric matrix.

\begin{lemma} \label{lem:sssvd}
Let $\mA = -\mA^T$ be an $n \times n$ skew-symmetric matrix
with eigenvalues $i \lambda_1, -i \lambda_1, i \lambda_2, -i \lambda_2, \ldots, 
i \lambda_j, -i\lambda_j$, where $\lambda_i > 0$ and $j = \lfloor n/2 \rfloor$.  
Then the SVD of $\mA$ is given by 
\begin{equation}
  \mA = \mU 
    \begin{sbmatrix} \lambda_1 \\ 
                    & \lambda_1 \\
                    & & \lambda_2 \\
                    & & & \lambda_2 \\
                    & & & & \ddots \\
                    & & & & & \lambda_j \\
                    & & & & & & \lambda_j 
    \end{sbmatrix} \mV^T
\end{equation}
for $\mU$ and $\mV$ given in the proof.
\end{lemma}
\begin{proof}
Using the Murnaghan-Wintner form of a real matrix 
\cite{murnaghan1931-canonical-form}, we can
write 
\[ \mA = \mX \mT \mX^T \]
for a \emph{real-valued} orthogonal matrix $\mX$ and
\emph{real-valued} block-upper-triangular
matrix $\mT$, with $2$-by-$2$ blocks along the diagonal.
Due to this form, $\mT$ must also be skew-symmetric.  Thus,
it is a block-diagonal matrix that 
we can permute to the form: 
\[ \mT = \begin{sbmatrix}
           0 & \lambda_1 \\
           -\lambda_1 & 0 \\
           & & 0 & \lambda_2 \\
           & & -\lambda_2 & 0 \\
           & & & & \ddots \\
         \end{sbmatrix}. \]
Note that the SVD of the matrix 
\[ \begin{bmatrix} 0 & \lambda_1 \\ -\lambda_1 & 0 \end{bmatrix} = 
\begin{bmatrix} 0 & 1 \\ 1 & 0 \end{bmatrix}
\begin{bmatrix} \lambda_1 & 0 \\ 0 & \lambda_1 \end{bmatrix}
\begin{bmatrix} -1 & 0 \\ 0 & 1 \end{bmatrix}.
\]
We can use this expression to complete the theorem:
\[
\begin{aligned}
\mA & = 
\underbrace{
	\mX 
	\begin{sbmatrix}
	   0 & 1 \\
	   1 & 0 \\
	   & & 0 & 1 \\
	   & & 1 & 0 \\
	   & & & & \ddots \\
	\end{sbmatrix}
}_{= \mU} 
\begin{sbmatrix} 
  	\lambda_1 \\ 
	& \lambda_1 \\
	& & \lambda_2 \\
	& & & \lambda_2 \\
	& & & & \ddots \\
\end{sbmatrix} 
\underbrace{
	\begin{sbmatrix}
	   -1 & 0 \\
	   0 & 1 \\
	   & & -1 & 0 \\
	   & & 0 & 1 \\
	   & & & & \ddots \\
	\end{sbmatrix}
	\mX^T
}_{= \mV^T}
. 
\end{aligned} \]
Both the matrices $\mU$ and $\mV$ are real and orthogonal.
Thus, this form yields the SVD of $\mA$.  
\end{proof}

We now use this lemma to show that -- under fairly general conditions --
the best rank-$k$ approximation to a skew-symmetric matrix is 
also skew-symmetric.
\begin{lemma} \label{lem:rank-k-sssvd}
Let $\mA$ be an $n$-by-$n$ skew-symmetric matrix, and let $k = 2j$ 
be even.  Let $\lambda_1 \ge \lambda_2 \ge \ldots \ge \lambda_j > \lambda_{j+1}$
be the magnitude of the singular value pairs.  (Recall that the
previous lemma showed that the singular values come in pairs.) Then the 
best rank-$k$ approximation of $\mA$ in an orthogonally invariant norm
is also skew-symmetric.
\end{lemma}
\begin{proof}
This lemma follows fairly directly from Lemma~\ref{lem:sssvd}.  Recall
that the best rank-$k$ approximation of $\mA$ in an orthogonally invariant
norm is given by the $k$ largest singular values and vectors.  By assumption
of the theorem, there is a gap in the spectrum between the $k$th and $k$+1-st
singular value.  Thus, taking the SVD form from Lemma~\ref{lem:sssvd}
and truncating to the $k$ largest singular values produces a skew-symmetric
matrix.  
\end{proof}

Finally, we can use this second result to show that the \algo{svp} algorithm
for the \algo{lasso} problem preserves skew-symmetry in all the 
iterates $\mX\itn{k}$.

\begin{theorem} \label{thm:sssvp}
 Given a set of skew-symmetric constraints $\sA(\cdot) = \vb$, the
 solution of the \algo{lasso} problem from the \algo{svp} solver is a skew-symmetric
 matrix $\mX$ if the target rank is even and the dominant singular
 values stay separated as in the previous lemma.
\end{theorem}
\begin{proof}
 In this proof, we revert to the notation $\sA(\mX)$ and use $\sA^*(\vz)$
 to denote the matrix with non-zeros in $\Omega$ and values from $\vz$.
 We proceed by induction on the iterates generated by the \algo{svp} algorithm.
 Clearly $\mX\itn{0}$ is skew-symmetric.  In step 3, we compute
 the SVD of a skew-symmetric matrix: $\sA^*(\sA(\mX\itn{k}) - \vb)$.
 The result, which is the next iterate, is skew-symmetric 
 based on the previous lemma and conditions of this theorem.  
\end{proof}

The \algo{svp} solver thus solves \eqref{eq:ranking-mc} for a fixed
rank problem.
A final step is to extract the scoring vector $\vs$ from a rank-$2$ singular
value decomposition.  If we had the exact matrix $\mY$, then
$(1/n) \mY \ve = \vs - (\vs^T \ve)/n \ve$, which yields
the score vector centered around 0.  Using a simple 
result noted by \citet{Meyer2010-Spread}, then
$\vs = (1/n) \mY \ve$
is also \emph{the best least-squares approximation} to $\vs$
in the case that $\mY$ is not an exact pairwise difference
matrix.  Formally, 
$(1/n) \mY \ve = 
\mathop{\mathrm{argmin}}_{\vs} 
\normof{ \mY^T - (\vs \ve^T - \ve \vs^T)}$.
The outcome that a rank-2 $\mU \mSigma \mV^T$ from
\algo{svp} is not of
the form $\vs \ve^T - \ve \vs^T$
 is quite possible because there are many
rank-2 skew-symmetric matrices that do not have $\ve$ 
as a factor.  However, the above discussion
justifies using $(1/n) \mU \mSigma \mV^T \vs$ 
derived from this
completed matrix.

Our complete ranking procedure is given by Algorithm~\ref{fig:nnra}.

\begin{algorithm}[t]
 \caption{Nuclear Norm Rank Aggregation. The \algo{svp}
 subroutine is given by Algorithm~\ref{fig:svp}.}
 \label{fig:nnra}
 \begin{algorithmic}[1]
   \REQUIRE ranking matrix $\mR$, minimum comparisons $c$
   \STATE Compute $\mY$ from $\mR$ by a
     procedure in Section~\ref{sec:pairwise}.
   \STATE Discard entries in $\mY$ with fewer than $c$ comparisons
   \STATE Let $\Omega$ be the index set for all retained entries in $\mY$ and 
      $\vb$ be the values for these entries
   \STATE $\mU,\mS,\mV$ = \algo{svp}(index set $\Omega$, values $\vb$, rank $2$)
   \STATE Compute $\vs = (1/n) \mU \mS \mV^T \ve$
 \end{algorithmic}
\end{algorithm}

\section{Other Approaches}

Now, we briefly compare our approach with other techniques to compute
ranking vectors from pairwise comparison data.  
An obvious approach is to find the least-squares solution 
$\min_{\vs} \sum_{(i,j) \in \Omega} (Y_{i,j} - (s_i - s_j))^2$.
This is a linear least squares method, and is exactly
what \citet{Massey-1997-Sports} proposed for ranking sports teams.
The related Colley method introduces a bit of regularization
into the least-squares problem~\cite{Colley-2002-bias}.
By way of comparison, the matrix completion approach has the
same ideal objective, however, we compute solutions using
a two-stage process: first complete the matrix, and then
extract scores.

A related methodology with skew-symmetric matrices
underlies recent developments in the 
application of Hodge theory to rank aggregation
 \cite{Jiang-2010-Hodge}.
By analogy with the Hodge decomposition of a vector space,
they propose a decomposition of pairwise rankings into
\emph{consistent}, \emph{globally inconsistent}, and \emph{locally inconsistent}
pieces. Our approach differs because our algorithm applies
without restriction on the comparisons.  
\citet{Freeman-1997-hierarchies} also uses an SVD of
a skew-symmetric matrix to discover a hierarchical
structure in a social network.

We know of two algorithms to directly estimate the item
value from ratings
\cite{dekerchov2007-dynamical-reputation,ho2008-ratings}.
Both of these methods include a technique to
model voter behavior.  They find that skewed
behaviors and inconsistencies in the ratings
require these adjustments.  In contrast, we
eliminate these problems by using the pairwise
comparison matrix.  Approaches
using a matrix or tensor factorization of
the rating matrix directly often have to
determine a rank empirically
\cite{Rendel2009-learning-rankings}.

The problem with the mean rating from Netflix
in Table~\ref{tab:netflix-top-movies}
is often corrected by requiring a minimum
number of rating on an item.  For example,
IMDB builds its top-250 movie list based
on a Bayesian estimate of the mean with at
least $3000$ ratings (\url{imdb.com/chart/top}).
Choosing this parameter is problematic as it
directly excludes items.
In contrast, choosing the minimum number of
comparisons to support an entry in $\mY$
may be easier to justify.

\section{Recoverability} \label{sec:recoverability}

A hallmark of the recent developments on matrix completion
is the existence of theoretical \emph{recoverability} guarantees
(see \citet{candes2009-exact-completion}, for example).
These guarantees give conditions under which the solution to
the optimization problems posed in Section~\ref{sec:ranking-nn}
\emph{is or is nearby} the low-rank matrix from whence the
samples originated.  In this section, we apply a recent theoretical
insight into matrix completion based on operator bases to
our problem of recovering a scoring vector from a skew-symmetric
matrix~\cite{Gross2010-low-rank}. We only treat the noiseless
problem to present a simplified analysis.  Also, the notation in
this section differs slight from the rest of the manuscript, 
in order to match the statements in \citet{Gross2010-low-rank} 
better.  In particular, $\Omega$ is not necessarily the 
index set, $\imath$ represents $\sqrt{-1}$, and
most of the results are for the complex field.

The goal is this section is to apply Theorem 3 from \citet{Gross2010-low-rank}
to skew-symmetric matrices arising from score difference vectors.  We restate
that theorem for reference.
\begin{theorem}[Theorem 3, \citet{Gross2010-low-rank}]
\label{thm:recovery}
Let $\mA$ be a rank-$r$ Hermitian matrix with coherence $\nu$ with respect to an 
operator basis $\{\mW_i\}_{i=1}^{n^2}$.  Let $\Omega \subset [1, n^2]$
be a random set of size $|\Omega| > O(nr\nu(1+\beta)(\log n)^2)$. 
Then the solution of 
\[ \MINone{}{ \normof[*]{\mX} }{ \tr(\mX^* \mW_i) = \tr(\mA^* \mW_i) \quad i \in \Omega} \]
is unique and is equal to $\mA$ with probability at least $1 - n^{-3}$.
\end{theorem}
The definition of coherence follows shortly.  On the surface,
this theorem is useless for our application.  The matrix we wish
to complete is not Hermitian, it's skew-symmetric.  
However, given a real-valued
skew-symmetric matrix $\mY$, the matrix $\imath \mY$ is Hermitian;
and hence, we will work to apply this theorem to this 
particular Hermitian matrix.
Again, we adopt this approach for simplicity.  It is likely
that a statement of Theorem~\ref{thm:recovery} 
with Hermitian replaced with
skew-Hermitian also holds, although verifying this would
require a reproduction of the proof from \citet{Gross2010-low-rank}.
 
The following theorem gives us a condition
for recovering the score vector using matrix completion.
As stated, this theorem is not particularly useful 
because $\vs$ may be recovered from noiseless measurements
by exploiting the special structure of the rank-2
matrix $\mY$.  For example, if we know $Y_{i,j} = s_i - s_j$ 
then given $s_i$ we can find $s_j$.  This argument may
be repeated with an arbitrary starting point as long 
as the known index set corresponds to a connected set
over the indices.  
Instead we view the following theorem as
providing intuition for the noisy problem.  

Consider the operator basis for Hermitian matrices:
\[ \begin{aligned}
\mathcal{H} & = \mathcal{S} \cup \mathcal{K} \cup \mathcal{D} \text{ where } \\
\mathcal{S} & = \{ 1/\sqrt{2}(\ve_i \ve_j^T + \ve_j \ve_i^T) : 1 \le i < j \le n \}; \\
\mathcal{K} & = \{ \imath/\sqrt{2}(\ve_i \ve_j^T - \ve_j \ve_i^T) : 1 \le i < j \le n \}; \\
\mathcal{D} & = \{ \ve_i \ve_i^T : 1 \le i \le n \}.
\end{aligned} \] 
\begin{theorem}
\label{thm:score-recovery}
Let $\vs$ be centered, i.e., $\vs^T \ve = 0$.
Let $\mY = \vs \ve^T - \ve \vs^T$ where  
$\theta = \max_i s_i^2/(\vs^T \vs)$ and
   $\rho = ((\max_i s_i) - (\min_i s_i))/\nnormof{\vs}$.
Also, let $\Omega \subset \mathcal{H}$ be a random set of 
elements with size $|\Omega| \ge O(2n\nu (1+\beta) (\log n)^2)$
where $\nu = \max( (n\theta + 1)/4, n \rho^2 )$. 
Then the solution of 
\[ \MINone{}{\normof[*]{\mX}}{%
   \tr(\mX^* \mW_i) = \tr((\imath \mY)^* \mW_i), \quad \mW_i \in \Omega}
\]
is equal to $\imath \mY$ with probability at least $1-n^{-\beta}$.	
\end{theorem}
The proof of this theorem follows directly by
Theorem~\ref{thm:recovery} if $\imath \mY$ has coherence
$\nu$ with respect to the basis $\mathcal{H}$.
We now show this result.
\begin{definition}[Coherence, \citet{Gross2010-low-rank}] 
Let $\mA$ be $n \times n$, rank-$r$,
and Hermitian.  Let $\mU\mU^*$ be an orthogonal projector
onto $\mathop{\mathrm{range}}(\mA)$.
Then $\mA$ has
coherence $\nu$ with respect to an operator basis 
$\{ \mW_i \}_{i=1}^{n^2}$ if both
\[ \begin{aligned}
\max\nolimits_i \tr(\mW_i \mU \mU^* \mW_i) & \le 2 \nu r/n, \text{ and} \\
\max\nolimits_i \tr(\sign(\mA) \mW_i)^2 & \le \nu r/n^2. \\
\end{aligned} \]
\end{definition}

For $\mA = \imath \mY$ with $\vs^T \ve = 0$:
\[ 
\mU\mU^* = \frac{\vs \vs^T}{\vs^T \vs} - \frac{1}{n} \ve \ve^T \text{ and }
\sign(\mA) = \frac{1}{\normof{\vs}\sqrt{n}} \mA.
\]
Let $\mS_p \in \mathcal{S}$, $\mK_p \in \mathcal{K}$, and 
$\mD_p \in \mathcal{D}$.  Note that 
because $\sign(\mA)$ is Hermitian with no real-valued
entries, both
quantities $\tr(\sign(\mA) \mD_i)^2 $ and
$\tr(\sign(\mA) \mS_i)^2 $ are $0$.
Also, because $\mU \mU^*$ is symmetric, 
$\tr(\mK_i \mU \mU^* \mK_p) = 0$.
The remaining
basis elements satisfy:
\[ 
\begin{aligned}
\tr(\mS_p \mU \mU^* \mS_p) & = \frac{1}{n} + \frac{s_i^2 + s_j^2}{2 \vs^T \vs} 
	\le (1/n) + \theta\\
\tr(\mD_p \mU \mU^* \mD_p) & = \frac{1}{n} + \frac{s_i^2}{\vs^T \vs} 
	\le (1/n) + \theta\\	
\tr(\sign(\mA) \mK_p)^2 &= \frac{2(s_i - s_j)^2}{n \vs^T \vs} \le (2/n)\rho^2.\\
\end{aligned}
\]
Thus, $\mA$ has coherence $\nu$
with $\nu$ from Theorem~\ref{thm:score-recovery}
and with respect to $\mathcal{H}$.  And we have 
our recovery result.  Although, 
this theorem provides little practical benefit unless
both $\theta$ and $\rho$ are $O(1/n)$,
which occurs when $\vs$ is nearly uniform.

\section{Results} \label{sec:results}
We implemented and tested this procedure in two synthetic
scenarios, along with Netflix, movielens, and Jester
joke-set ratings data.  In the interest of space, we only present
a subset of these results for Netflix.

\subsection{Recovery} \label{sec:recovery}

\begin{figure}
\includegraphics[width=0.49\linewidth]{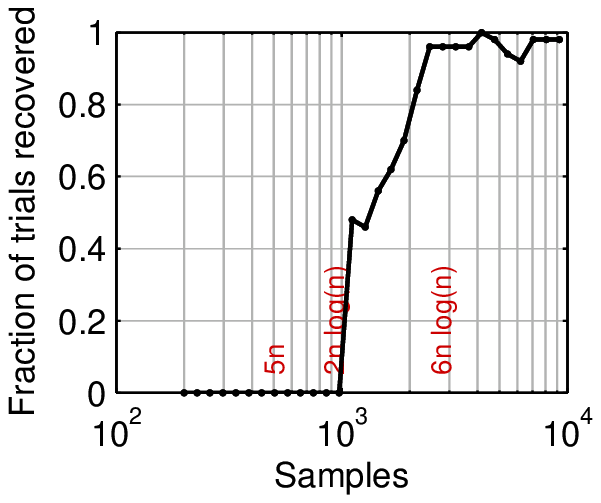}
\includegraphics[width=0.49\linewidth]{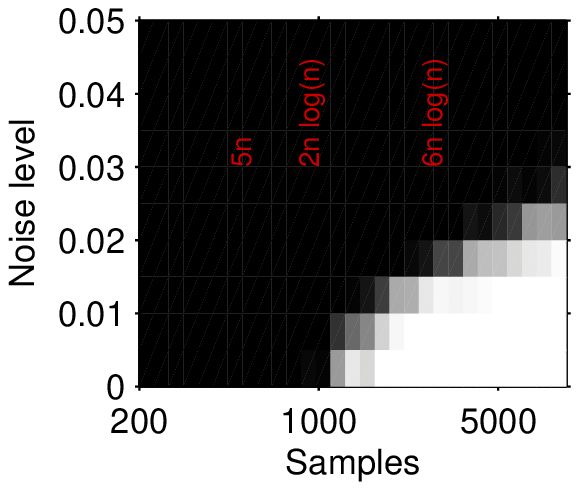}
\caption{An experimental study of the recoverability
of a ranking vector.  These show that we need
about 6n log n entries of $\mY$ to get good recovery
in both the noiseless (left) and noisy (right) case.
See \S\ref{sec:recovery} for more information.}
\label{fig:recovery}
\end{figure}

The first experiment is an empirical study
of the recoverability of the score vector in
the noiseless and noisy case.  
In the noiseless case, Figure~\ref{fig:recovery} (left),
we generate
a score vector with uniformly distributed random scores between
0 and 1.  These are used to construct
a pairwise comparison matrix $\mY = \vs \ve^T - \ve \vs^T$.
We then sample elements of this matrix uniformly at
random and compute the difference between
the true score vector $\vs$ and the output of
steps 4 and 5 of Algorithm~\ref{fig:nnra}.  If the 
relative $2$-norm
difference between these vectors is less than $10^{-3}$,
we declare the trial recovered.  For 
$n=100$, the figure shows that, once
the number of samples is about $6n \log n$, 
the correct $\vs$ is recovered in nearly 
all the 50 trials.

Next, for the noisy case, we generate 
a uniformly spaced score vector between
0 and 1.  Then $\mY = \vs \ve^T - \ve \vs^T + \eps \mE$,
where $\mE$ is a matrix of random normals.  
Again, we sample elements of this matrix randomly,
and declare a trial successful if the \emph{order}
of the recovered score vector is identical
to the true order.  
In Figure~\ref{fig:recovery} (right),
we indicate the fractional of successful trials
as a gray value between black (all failure) 
and white (all successful).  Again, the algorithm
is successful for a moderate noise level, i.e.,
the value of $\eps$, when the number of samples
is larger than $6n \log n$.

\begin{figure}
\includegraphics[width=0.49\linewidth]{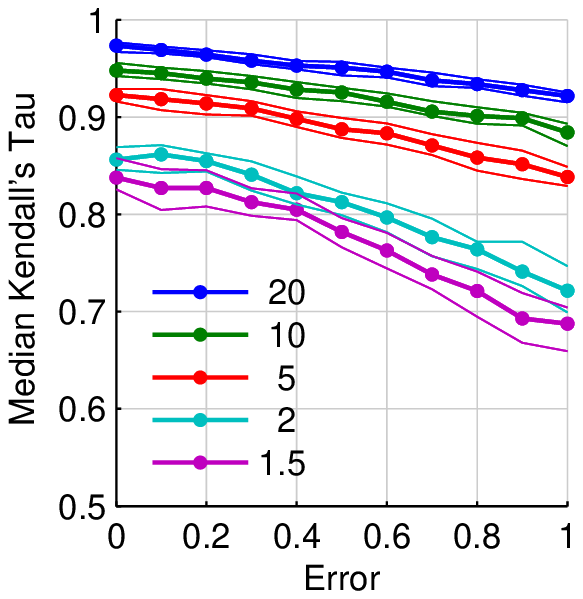}
\includegraphics[width=0.49\linewidth]{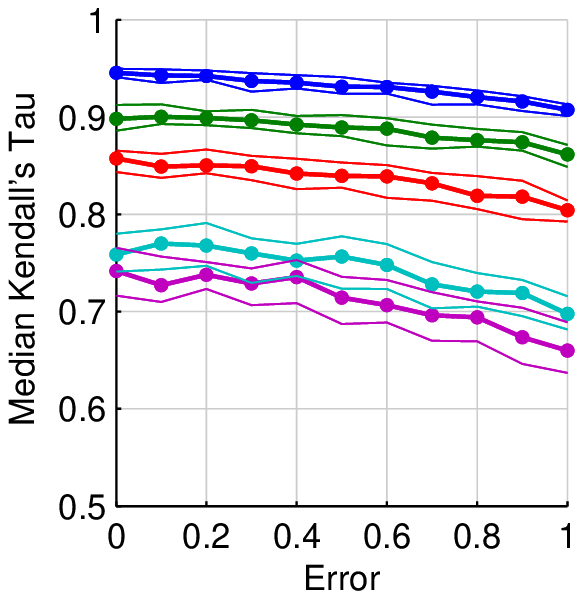}
\caption{The performance of our algorithm (left)
and the mean rating (right) to recovery
the ordering given by item scores 
in an item-response theory model
with 100 items and 1000 users.  
The various thick lines 
correspond to average number of ratings 
each user performed (see the in place legend).
See \S\ref{sec:synthetic} for more information}
\label{fig:synthetic}
\end{figure}

\subsection{Synthetic} \label{sec:synthetic}

Inspired by \citet{ho2008-ratings}, we investigate
recovering item scores in an item-response scenario.
Let $a_i$ be the center of user $i$'s rating scale,
and $b_i$ be the rating sensitivity of user $i$.
Let $t_i$ be the intrinsic score of item $j$.  
Then we generate ratings from users on items
as: 
\[ R_{i,j} = L[ a_i + b_i t_j + E_{i,j} ] \]
where $L[\alpha]$ is the discrete levels function: 
\[ L[\alpha] = 
\begin{cases} 
  1 & \alpha < 1.5 \\
	2 & 1.5 \le \alpha < 2.5 \\
	3 & 2.5 \le \alpha < 3.5 \\
	4 & 3.5 \le \alpha < 4.5 \\
	5 & 4.5 \le \alpha, \end{cases} \]
and $E_{i,j}$ is a noise parameter.	
In our experiment, we draw $a_i \sim N(3,1)$, 
$b_i \sim N(0.5, 0.5)$, $t_i \sim N(0.1, 1)$,
and $E_{i,j} \sim \eps N(0,1)$.  Here,
$N(\mu,\sigma)$ is a standard normal, and
$\eps$ is a noise
parameter.  As input to our algorithm, 
we sample ratings uniformly at random by
specifying a desired number of average ratings
per user.  We then look at
the Kendall $\tau$ correlation coefficient
between the true scores $t_i$ and the output
of our algorithm using the arithmetic mean
pairwise aggregation.  A $\tau$ value of
1 indicates a perfect ordering correlation
between the two sets of scores.

Figure~\ref{fig:synthetic} shows the results
for $1000$ users and $100$ items with 
$1.1,1.5, 2, 5, $ and $10$ ratings 
per user on average.  We also vary the
parameter $\eps$ between $0$ and $1$.
Each thick line with markers plots the 
median value of $\tau$ in 50 trials.  The
thin adjacency lines show the $25$th and
$75$th percentiles of the 50 trials.  
At all error levels, our algorithm
outperforms the mean rating.  Also,
when there are few ratings per-user
and moderate noise, 
our approach is considerably more
correlated with the true score.  
This evidence supports the
anecdotal results from Netflix in 
Table~\ref{tab:netflix-top-movies}.

\subsection{Netflix}
\label{sec:netflix}

See Table~\ref{tab:netflix-top-movies}
for the top movies produced by our technique in a few circumstances
using all users.  The arithmetic mean results in that table
use only elements of $\mY$ with at least $30$ pairwise comparisons
(it is a \texttt{am all 30} model in the code below).
And see Figure~\ref{fig:netflix-residuals} for an analysis
of the residuals generated by the fit for different constructions
of the matrix $\mYhat$.  
Each residual evaluation of Netflix is described by a code.
For example, \texttt{sb all 0} is a strict-binary pairwise
matrix $\mYhat$ from all Netflix users and $c=0$ in Algorithm~\ref{fig:nnra}
(i.e. accept all pairwise comparisons).  Alternatively,
\texttt{am 6 30} denotes an arithmetic-mean pairwise matrix $\mYhat$
from Netflix users with at least 6 ratings, where each entry in $\mYhat$
had 30 users supporting it.  The other abbreviations are
\texttt{gm}: geometric mean; \texttt{bc}: binary comparison; and
\texttt{lo}: log-odds ratio.

These residuals show that we get better rating
fits by only using frequently compared movies, but that there 
are only minor changes in the fits when excluding users that
rate few movies.  The difference between
the score-based residuals $\normof{\Omega(\vs \ve^T - \ve
\vs^T)-\vb}$ (red points) and the \algo{svp} residuals 
$\normof{\Omega(\mU \mS \mV^T)-\vb}$ (blue points) show that 
excluding comparisons leads to ``overfitting'' 
in the \algo{svp} residual.  This suggests that increasing
the parameter $c$ should be done with care and good
checks on the residual norms.

\enlargethispage{\baselineskip}
To check that a rank-$2$ approximation is reasonable,
we increased the target rank in the \algo{svp} solver
to $4$ to investigate.  For the arithmetic mean (6,30) 
model, the relative residual at rank-$2$ is $0.2838$ and
at rank-$4$ is $0.2514$.  Meanwhile, the nuclear
norm increases from around 14000 to around 17000.
These results show that the change in the fit is minimal
and our rank-2 approximation and its scores
should represent a reasonable ranking.

 \begin{figure}[t]
  \centering
  \includegraphics[width=0.8\linewidth]{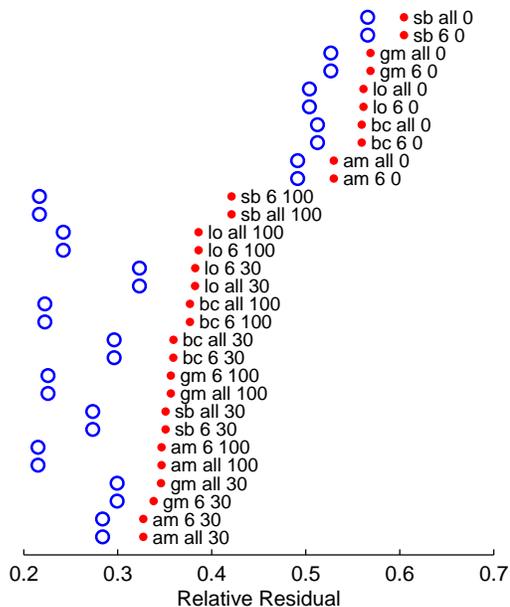}
  \caption{The labels on each residual show the method to generate
   the pairwise scores and how we truncated the Netflix data. 
	 Red points are the residuals from the scores, and blue 
	 points are the final residuals from the SVP algorithm.
	 Please see the discussion in \S\ref{sec:netflix}.} 
   \label{fig:netflix-residuals}
	 \vspace{-\baselineskip}
 \end{figure}

\section{Conclusion}


Existing principled techniques such
as computing a Kemeny optimal ranking or finding a minimize feedback 
arc set are NP-hard.  These approaches
are inappropriate in large scale rank aggregation settings.
Our proposal is (i) measure pairwise scores $\mYhat$ and (ii)
solve a matrix completion problem to determine the quality
of items.  This idea is both principled and functional with
significant missing data.  The results of our rank aggregation
on the Netflix problem (Table~\ref{tab:netflix-top-movies}) reveal
popular and high quality movies.  These are interesting results
and could easily have a home on a ``best movies in Netflix''
web page.  Computing a rank aggregation
with this technique is not NP-hard.  It only requires
solving a convex optimization problem with a unique global
minima.  Although we did not record computation times,
the most time consuming piece of work is computing the pairwise
comparison matrix $\mY$.  In a practical setting, this could
easily be done with a MapReduce computation.

To compute these solutions, we adapted the \algo{svp} solver
for matrix completion \cite{Jain-2010-SVP}.  This process
involved (i) studying the singular value decomposition of a 
skew-symmetric matrix (Lemmas~\ref{lem:sssvd} and \ref{lem:rank-k-sssvd})
and (ii) showing that the \algo{svp} solver preserves a skew-symmetric
approximation through its computation (Theorem~\ref{thm:sssvp}).
Because the \algo{svp} solver computes with an explicitly chosen rank,
these techniques work well for large scale rank aggregation problems.

We believe the combination of pairwise aggregation and matrix completion
is a fruitful direction for future research.  We
plan to explore optimizing the \algo{svp} algorithm to exploit
the skew-symmetric constraint,
extending our recovery result to the noisy case,
and investigating additional data.

\let\thefootnote\relax
\footnotetext{\textit{Funding.} \scriptsize
David F.\ Gleich was supported in 
part by the Natural Sciences and Engineering Research
 Council of Canada along with the Department of Energy's 
John von Neumann fellowship.
Lek-Heng Lim acknowledges the support of NSF CAREER Award DMS 1057064.
}

\bibliography{all-bibliography}

\begin{thebibliography}{38}
\providecommand{\natexlab}[1]{#1}
\providecommand{\url}[1]{\texttt{#1}}
\expandafter\ifx\csname urlstyle\endcsname\relax
  \providecommand{\doi}[1]{doi: #1}\else
  \providecommand{\doi}{doi: \begingroup \urlstyle{rm}\Url}\fi
  \scriptsize

\bibitem[Ailon et~al.(2005)Ailon, Charikar, and Newman]{ailon2005-ranking}
N.~Ailon, M.~Charikar, and A.~Newman.
\newblock Aggregating inconsistent information: ranking and clustering.
\newblock In \emph{STOC '05}, pages 684--693, 2005.
\newblock ISBN 1-58113-960-8.
\newblock \doi{10.1145/1060590.1060692}.

\bibitem[Alon(2006)]{alon2006-ranking}
N.~Alon.
\newblock Ranking tournaments.
\newblock \emph{SIAM J. Discret. Math.}, 20\penalty0 (1):\penalty0 137--142,
  2006.
\newblock ISSN 0895-4801.
\newblock \doi{10.1137/050623905}.

\bibitem[Arrow(1950)]{arrow1950-impossibility}
K.~J. Arrow.
\newblock A difficulty in the concept of social welfare.
\newblock \emph{J. Polit. Econ.}, 58\penalty0 (4):\penalty0 328--346, August
  1950.
\newblock ISSN 00223808.
\newblock URL \url{http://www.jstor.org/stable/1828886}.

\bibitem[Cai et~al.(2008)Cai, Cand\`{e}s, and Shen]{cai2008-svt}
J.-F. Cai, E.~J. Cand\`{e}s, and Z.~Shen.
\newblock A singular value thresholding algorithm for matrix completion.
\newblock \emph{arXiv}, math.OC:\penalty0 0810.3286v1, October 2008.
\newblock URL \url{http://arxiv.org/abs/0810.3286}.

\bibitem[Cand\`{e}s and Tao(2007)]{candes2007-dantzig}
E.~Cand\`{e}s and T.~Tao.
\newblock The {Dantzig} selector: Statistical estimation when $p$ is much
  larger than $n$.
\newblock \emph{Ann. Stat.}, 35\penalty0 (6):\penalty0 2313--2351, 2007.
\newblock \doi{10.1214/009053606000001523}.

\bibitem[Cand\`{e}s and Recht(2009)]{candes2009-exact-completion}
E.~J. Cand\`{e}s and B.~Recht.
\newblock Exact matrix completion via convex optimization.
\newblock \emph{Found. Comput. Math.}, 9\penalty0 (6):\penalty0 717--772,
  December 2009.
\newblock \doi{10.1007/s10208-009-9045-5}.

\bibitem[Cand\`{e}s and Tao(to appear)]{candes2009-matrix-completion}
E.~J. Cand\`{e}s and T.~Tao.
\newblock The power of convex relaxation: Near-optimal matrix completion.
\newblock \emph{IEEE Trans. Inform. Theory}, to appear.

\bibitem[Chen et~al.(1998)Chen, Donoho, and Saunders]{chen1998-atomic}
S.~S. Chen, D.~L. Donoho, and M.~A. Saunders.
\newblock Atomic decomposition by basis pursuit.
\newblock \emph{SIAM J. Sci. Comp.}, 20\penalty0 (1):\penalty0 33--61, 1998.
\newblock \doi{10.1137/S1064827596304010}.

\bibitem[Colley(2002)]{Colley-2002-bias}
W.~N. Colley.
\newblock Colley's bias free college football ranking method: The {Colley}
  matrix explained.
\newblock Technical report, Princeton University, 2002.

\bibitem[Condorcet(1785)]{condorcet1785-essai}
J.-A.-N. d.~C. Condorcet.
\newblock \emph{Essai sur l'application de l'analyse \`{a} la probabilit\'{e}
  des d\'{e}cisions...}
\newblock de L'imprimerie Royale, Paris, 1785.
\newblock URL \url{http://gallica2.bnf.fr/ark:/12148/bpt6k417181}.

\bibitem[Dai and Milenkovic(2009)]{dai2009-set}
W.~Dai and O.~Milenkovic.
\newblock Set: an algorithm for consistent matrix completion.
\newblock \emph{arXiv}, September 2009.
\newblock URL \url{http://arxiv.org/abs/0909.2705}.

\bibitem[David(1988)]{david1988-paired}
H.~A. David.
\newblock \emph{The method of paired comparisons}.
\newblock Number~41 in Griffin's Statistical Monographs and Courses. Charles
  Griffin, 1988.
\newblock ISBN 0195206169.

\bibitem[de~Kerchov and van Dooren(2007)]{dekerchov2007-dynamical-reputation}
C.~de~Kerchov and P.~van Dooren.
\newblock Iterative filtering for a dynamical reputation system.
\newblock \emph{arXiv}, cs.IR:\penalty0 0711.3964, 2007.
\newblock URL \url{http://arXiv.org/abs/0711.3964}.

\bibitem[Dwork et~al.(2001)Dwork, Kumar, Naor, and
  Sivakumar]{dwork2001-rank-aggregation}
C.~Dwork, R.~Kumar, M.~Naor, and D.~Sivakumar.
\newblock Rank aggregation methods for the web.
\newblock In \emph{WWW '01}, pages 613--622, New York, NY, USA, 2001. ACM.
\newblock ISBN 1-58113-348-0.
\newblock \doi{10.1145/371920.372165}.

\bibitem[Fazel(2002)]{fazel2002-phdthesis}
M.~Fazel.
\newblock \emph{Matrix rank minimization with applications}.
\newblock PhD thesis, Stanford University, March 2002.
\newblock URL \url{http://faculty.washington.edu/mfazel/thesis-final.pdf}.

\bibitem[Freeman(1997)]{Freeman-1997-hierarchies}
L.~C. Freeman.
\newblock Uncovering organizational hierarchies.
\newblock \emph{Computational and Mathematical Organization Theory}, 3\penalty0
  (1):\penalty0 5--18, 1997.
\newblock ISSN 1381-298X.
\newblock \doi{10.1023/A:1009690520577}.

\bibitem[Fuchs(2004)]{fuchs2004-noise}
J.-J. Fuchs.
\newblock Recovery of exact sparse representations in the presence of noise.
\newblock In \emph{ICASSP '04}, volume~2, pages ii--533--6 vol.2, May 2004.
\newblock \doi{10.1109/ICASSP.2004.1326312}.

\bibitem[Gross(2010)]{Gross2010-low-rank}
D.~Gross.
\newblock Recovering low-rank matrices from few coefficients in any basis.
\newblock \emph{arXiv}, cs.NA:\penalty0 0910.1879v5, 2010.
\newblock URL \url{http://arxiv.org/abs/0910.1879}.

\bibitem[Ho and Quinn(2008)]{ho2008-ratings}
D.~E. Ho and K.~M. Quinn.
\newblock Improving the presentation and interpretation of online ratings data
  with model-based figures.
\newblock \emph{Amer. Statist.}, 62\penalty0 (4):\penalty0 279--288, November
  2008.
\newblock \doi{10.1198/000313008X366145}.
\newblock URL \url{http://pubs.amstat.org/doi/abs/10.1198/000313008X366145}.

\bibitem[Jain et~al.(2010)Jain, Meka, and Dhillon]{Jain-2010-SVP}
P.~Jain, R.~Meka, and I.~Dhillon.
\newblock Guaranteed rank minimization via singular value projection.
\newblock In J.~Lafferty, C.~K.~I. Williams, J.~Shawe-Taylor, R.~Zemel, and
  A.~Culotta, editors, \emph{Advances in Neural Information Processing Systems
  23}, pages 937--945, 2010.
\newblock URL \url{http://books.nips.cc/papers/files/nips23/NIPS2010_0682.pdf}.

\bibitem[Jiang et~al.(2010)Jiang, Lim, Yao, and Ye]{Jiang-2010-Hodge}
X.~Jiang, L.-H. Lim, Y.~Yao, and Y.~Ye.
\newblock Statistical ranking and combinatorial hodge theory.
\newblock \emph{Mathematical Programming}, 127\penalty0 (1):\penalty0 1--42,
  2010.
\newblock ISSN 0025-5610.
\newblock \doi{10.1007/s10107-010-0419-x}.
\newblock 10.1007/s10107-010-0419-x.

\bibitem[Kemeny(1959)]{kemeny1959-math-without-numbers}
J.~G. Kemeny.
\newblock Mathematics without numbers.
\newblock \emph{Daedalus}, 88\penalty0 (4):\penalty0 577--591, Fall 1959.
\newblock ISSN 00115266.
\newblock URL \url{http://www.jstor.org/stable/20026529}.

\bibitem[Kendall and Smith(1940)]{Kendall-1940-paired-comparison}
M.~G. Kendall and B.~B. Smith.
\newblock On the method of paired comparison.
\newblock \emph{Biometrika}, 31\penalty0 (3-4):\penalty0 324--345, 1940.
\newblock \doi{10.1093/biomet/31.3-4.324}.

\bibitem[Keshavan and Oh(2009)]{keshavan2009-grassman}
R.~H. Keshavan and S.~Oh.
\newblock A gradient descent algorithm on the grassman manifold for matrix
  completion.
\newblock \emph{arXiv}, October 2009.
\newblock URL \url{http://arxiv.org/abs/0910.5260}.

\bibitem[Langville and Meyer(forthcoming)]{Meyer2010-Spread}
A.~N. Langville and C.~D. Meyer.
\newblock \emph{Who's \#1:The Science of Rating and Ranking}.
\newblock Princeton University Press, Princeton, NJ, forthcoming.

\bibitem[Lee and Bresler(2009)]{lee2009-admira}
K.~Lee and Y.~Bresler.
\newblock Admira: Atomic decomposition for minimum rank approximation.
\newblock \emph{arXiv}, May 2009.
\newblock URL \url{http://arxiv.org/abs/0905.0044}.

\bibitem[Li et~al.(2008)Li, Liu, and Zhai]{li2008-learning-to-rank}
H.~Li, T.-Y. Liu, and C.~Zhai, editors.
\newblock \emph{Proceedings of the SIGIR 2008 Workshop: Learning to Rank for
  Information Retrieval}.
\newblock 2008.
\newblock URL
  \url{http://research.microsoft.com/en-us/um/beijing/events/lr4ir-2008/PROCEEDINGS-LR4IR%202008.PDF}.

\bibitem[Massey(1997)]{Massey-1997-Sports}
K.~Massey.
\newblock Statistical models applied to the rating of sports teams.
\newblock Master's thesis, Bluefield College, 1997.

\bibitem[Mazumder et~al.(2009)Mazumder, Hastie, and
  Tibshirani]{mazumder2009-regularization}
R.~Mazumder, T.~Hastie, and R.~Tibshirani.
\newblock Regularization methods for learning incomplete matrices.
\newblock \emph{arXiv}, June 2009.
\newblock URL \url{http://arxiv.org/abs/0906.2034v1}.

\bibitem[Miller(1956)]{miller1956-seven}
G.~A. Miller.
\newblock The magical number seven, plus or minus two: Some limits on our
  capacity for processing information.
\newblock \emph{Psychol. Rev.}, 101\penalty0 (2):\penalty0 343--352, 1956.
\newblock URL
  \url{http://www.psych.utoronto.ca/users/peterson/psy430s2001/Miller%20GA%20Magical%20Seven%20Psych%20Review%201955.pdf}.

\bibitem[Murnaghan and Wintner(1931)]{murnaghan1931-canonical-form}
F.~D. Murnaghan and A.~Wintner.
\newblock A canonical form for real matrices under orthogonal transformations.
\newblock \emph{PNAS}, 17\penalty0 (7):\penalty0 417--420, July 1931.
\newblock URL \url{http://www.pnas.org/content/17/7/417.full.pdf+html}.

\bibitem[Recht et~al.(to appear)Recht, Fazel, and
  Parrilo]{recht2009-nuclear-norm}
B.~Recht, M.~Fazel, and P.~A. Parrilo.
\newblock Guaranteed minimum-rank solution of linear matrix equations via
  nuclear norm minimization.
\newblock \emph{SIAM Rev.}, to appear.

\bibitem[Rendle et~al.(2009)Rendle, Balby~Marinho, Nanopoulos, and
  Schmidt-Thieme]{Rendel2009-learning-rankings}
S.~Rendle, L.~Balby~Marinho, A.~Nanopoulos, and L.~Schmidt-Thieme.
\newblock Learning optimal ranking with tensor factorization for tag
  recommendation.
\newblock In \emph{KDD '09: Proceedings of the 15th ACM SIGKDD international
  conference on Knowledge discovery and data mining}, pages 727--736, New York,
  NY, USA, 2009. ACM.
\newblock ISBN 978-1-60558-495-9.
\newblock \doi{10.1145/1557019.1557100}.

\bibitem[Saaty(1987)]{saaty1987-perron}
T.~L. Saaty.
\newblock Rank according to {Perron}: A new insight.
\newblock \emph{Math. Mag}, 60\penalty0 (4):\penalty0 211--213, October 1987.
\newblock ISSN 0025570X.
\newblock URL \url{http://www.jstor.org/stable/2689340}.

\bibitem[Tan and Jin(2004)]{Tan-2004-ordering}
P.-N. Tan and R.~Jin.
\newblock Ordering patterns by combining opinions from multiple sources.
\newblock In \emph{Proceedings of the tenth ACM SIGKDD international conference
  on Knowledge discovery and data mining}, KDD '04, pages 695--700, New York,
  NY, USA, 2004. ACM.
\newblock ISBN 1-58113-888-1.
\newblock \doi{http://doi.acm.org/10.1145/1014052.1014142}.
\newblock URL \url{http://doi.acm.org/10.1145/1014052.1014142}.

\bibitem[Tibshirani(1996)]{tibshirani1996-lasso}
R.~Tibshirani.
\newblock Regression shrinkage and selection via the lasso.
\newblock \emph{J. R. Stat. Soc. Ser. B Stat. Methodol.}, 58\penalty0
  (1):\penalty0 267--288, 1996.
\newblock ISSN 00359246.
\newblock URL \url{http://www.jstor.org/stable/2346178}.

\bibitem[Toh and Yun(2009)]{toh2009-proximal-gradient}
K.-C. Toh and S.~Yun.
\newblock An accelerated proximal gradient algorithm for nuclear norm
  regularized linear least squares problems.
\newblock \emph{Opt. Online}, November 2009.
\newblock URL
  \url{http://www.optimization-online.org/DB_FILE/2009/03/2268.pdf}.

\bibitem[Vandenberghe and Boyd(1996)]{vandenberghe1996-semidefinite}
L.~Vandenberghe and S.~Boyd.
\newblock Semidefinite programming.
\newblock \emph{SIAM Rev.}, 38\penalty0 (1):\penalty0 49--95, March 1996.
\newblock ISSN 0036-1445.
\newblock \doi{10.1137/1038003}.

\end{thebibliography}
\bibliographystyle{abbrvnat}

\end{document}